\documentclass[letterpaper, 10 pt, conference]{ieeeconf}  

\IEEEoverridecommandlockouts                              

\usepackage{graphics} 
\usepackage{epsfig} 
\usepackage{mathptmx} 
\usepackage{times} 
\usepackage{amsmath} 
\usepackage{amssymb}  
\usepackage{subfigure}
\usepackage{color}
\usepackage{flushend}
\usepackage{eurosym}
\usepackage{amsmath,theorem}
\usepackage{amssymb}
\usepackage{amsfonts}
\usepackage{graphics}
\usepackage{psfrag}
\usepackage{array}
\usepackage{shortvrb}
\usepackage{epsf}
\usepackage{graphicx}
\usepackage{rotating}
\usepackage{cite}
\usepackage{wasysym}
\usepackage[usenames]{xcolor}

\usepackage{graphicx}
\usepackage{eurosym}
\usepackage{amssymb}
\usepackage{amsmath,mathtools}
\usepackage{amsfonts}
\usepackage{epstopdf}
\usepackage{epsf,subfigure}
\usepackage{psfrag}
\usepackage{graphics}
\usepackage{color} 

\newtheorem{definition}{Definition}
\newtheorem{theorem}{Theorem}

\newtheorem{lemma}{Lemma}

\newcommand{\qed}{\nobreak \ifvmode \relax \else
      \ifdim\lastskip<1.5em \hskip-\lastskip
      \hskip1.5em plus0em minus0.5em \fi \nobreak
      \vrule height0.75em width0.5em depth0.25em\fi}

\title{\LARGE \bf
A Sum-of-Squares approach to the Stability  and Control of Interconnected Systems using Vector Lyapunov Functions*
}

\author{Soumya Kundu$^{1}$ and Marian Anghel$^{2}$
\thanks{*This work was supported by the U.S. Department of Energy
through the LANL/LDRD Program.}
\thanks{$^{1}$Soumya Kundu is with the Center for Nonlinear Studies and Information Sciences Group (CCS-3), Los Alamos National Laboratory, Los Alamos, USA
        {\tt\small soumya@lanl.gov}}%
\thanks{$^{2}$Marian Anghel is with the Information Sciences Group (CCS-3), Los Alamos National Laboratory, Los Alamos, USA
        {\tt\small manghel@lanl.gov}}%
}

\begin{document}

\maketitle
\thispagestyle{empty}
\pagestyle{empty}

\begin{abstract}

Stability analysis tools are essential to understanding and controlling any engineering system. Recently{,} sum-of-squares (SOS) based methods have been used to compute Lyapunov based estimates {for the} region-of-attraction (ROA) of polynomial dynamical systems. But for a real-life large scale dynamical system this method becomes inapplicable because of growing computational burden. In such a case, it is important to develop a subsystem based stability analysis approach which is the focus of the work presented here. A parallel and scalable algorithm is used to infer stability of an interconnected system, with the help of the subsystem Lyapunov functions. Locally computable control laws are proposed to guarantee asymptotic stability under a given disturbance.

\end{abstract}

\section{INTRODUCTION}


Toward the end of the nineteenth century the Russian mathematician A.~M.~Lyapunov \cite{Lyapunov:1892} has  introduced a number of powerful tools 
for the stability analysis of nonlinear dynamical systems. His stability  results have been later generalized by Barbashin, Krasovskii, and LaSalle, 
while control system engineers have used Lyapunov's methods to design stabilizing feedback controllers --- see \cite{Haddad:2008} and references therein.
Nevertheless, constructing a system Lyapunov function is usually a difficult task, which becomes daunting when the size of the dynamical system increases.

{
A more practical approach is to define the Lyapunov function of the interconnected system as some function of the subsystem Lyapunov functions. There are different functional forms for the Lyapunov function of the interconnected system, such as a scalar Lyapunov function expressed as a weighted sum of the subsystem Lyapunov functions, or applications of vector Lyapunov functions and comparison principles \cite{Siljak:1972,Weissenberger:1973, Michel:1983,Araki:1978}. Formulations using vector Lyapunov functions \cite{Bellman:1962,Bailey:1966} are computationally attractive because of their parallel structure and scalability. In \cite{Weissenberger:1973}, it was shown that if the subsystem Lyapunov functions and the interactions satisfy certain conditions, then application of comparison equations \cite{Conti:1956,Brauer:1961,Beckenbach:1961} can provide a certificate of exponential stability of the interconnected systems. 
In this work we seek an algorithmic certification of asymptotic stability via the vector Lyapunov function approach, where each subsystem Lyapunov functions are expressed in some polynomial form.}

 {P}{rimarily} we will focus on an example of a randomly generated network of nine modified\footnote{We choose the oscillator parameters in such a way that the Van der Pol oscillators have a stable equilibrium at origin.} Van der Pol oscillators. Each Van der Pol oscillator can be represented as a two-state system with state dynamic equations as polynomials of degree three \cite{van1926}. The network is then decomposed into many interacting subsystems. Each subsystem parameters are so chosen that individually each subsystem is stable, when the disturbances from neighbors are zero. Sum-of-squares based expanding interior algorithm \cite{Wloszek:2003,Anghel:2013} is used to obtain estimate of region of attraction as sub-level sets of polynomial Lyapunov functions for each such subsystem. Finally a sum-of-squares based scalable and parallel algorithm is used to certify stability in the sense of Lyapunov of the interconnected system by using the subsystem Lyapunov functions computed in the previous step. A distributed control strategy is proposed that can guarantee asymptotic stability of the interconnected system under given disturbances. {Following some brief background in Sec.~\ref{S:background} we outline the problem statement in Sec.~\ref{S:problem}. An algorithmic approach to certifying asymptotic stability is presented in Sec.~\ref{S:stability} while a distributed control strategy is discussed in Sec.~\ref{S:control}. Sec.~\ref{S:results} shows an application of our stability analysis and control approach to a network of Van der Pol oscillators. We conclude the article in Sec.~\ref{S:conclusion}.}

\section{{BASIC CONCEPTS AND BACKGROUND}}
\label{S:background}
Before formulating the problem, let us briefly review some of the key concepts behind our analysis. We will first discuss how the stability of a dynamical system can be analyzed by constructing suitable Lyapunov functions. Then we briefly refer to sum-of-square polynomials and a very useful result which helps us in formulating the sum-of-squares problems. 

\subsection{Lyapunov Stability Methods}
\label{S:Lyap}
Let us consider the dynamical system 
\begin{align}\label{E:f}
&\dot{x}\left(t\right) = f\left(x\left(t\right)\right),\quad t\geq 0,~x\in\mathbb{R}^n, ~f\left(\mathbf{0}_{n\times 1}\right)=\mathbf{0}_{n\times 1}
\end{align}
We assume that $x=\mathbf{0}_{n\times 1}$ (for simplicity, $t$ will be often dropped when obvious) is an equilibrium point of the dynamical system\footnote{Note that this is not a restrictive assumption, since by shifting of state variables, the origin can always be made an equilibrium point.}, and $f:\mathbb{R}^n\rightarrow \mathbb{R}^n$ is locally Lipschitz. 
An important notion of stability is as follows:
\begin{definition}
The equilibrium point at origin is called asymptotically stable if it is stable in the sense of Lyapunov, and if 
\begin{align}
\exists \tilde{\delta}>0 ~\text{s.t.}~ \|x(0)\|_2<\tilde{\delta}\implies\lim_{t\rightarrow +\infty}\|x(t)\|_2=0.
\end{align}
\end{definition}

The Lyapunov stability theorem \cite{Lyapunov:1892,Slotine:1991} presents a sufficient condition of stability through the construction of a certain positive definite function.
\begin{theorem}\label{T:Lyap}
The equilbrium point $x=\mathbf{0}_{n\times 1}$ of the dynamical system in (\ref{E:f}) is stable in the sense of Lyapunov in $\mathcal{D}\in\mathbb{R}^n$, if there exists a continuously differentiable positive definite function {$\tilde{V}:\mathcal{D}\rightarrow \mathbb{R}$} (henceforth referred to as Lyapunov function) such that,
\begin{subequations}
\begin{align}
\tilde{V}\left(0\right) &= 0 \\
\tilde{V}\left(x\right) &>0,\forall x\in\mathcal{D}\backslash{\left\lbrace\mathbf{0}_{n\times 1}\right\rbrace}\\
\text{and, }-\dot{\tilde{V}}\left(x\right)&\geq 0,\forall x\in\mathcal{D}
\end{align}
\end{subequations}
Further, if $\tilde{V}$ satisfies $-\dot{\tilde{V}}(x)>0,\forall x\in\mathcal{D}\backslash{\left\lbrace\mathbf{0}_{n\times 1}\right\rbrace},$ then the equilibrium point at origin is asymptotically stable in $\mathcal{D}$.
\end{theorem}
When there exists such a function $\tilde{V}\left(x\right)$, the region of attraction (ROA) of the stable equilibrium point at origin can be (conservatively) estimated as
\begin{subequations}\label{E:ROA}
\begin{align}
\mathcal{R}_A&:=\left\lbrace x\in\mathcal{D}\left| \tilde{V}(x)\leq \gamma^{max}\right.\right\rbrace\\
\text{where,}~\gamma^{max}&:=\arg\max_\gamma\left\lbrace x\in\mathbb{R}^n\left| \tilde{V}(x)\leq\gamma\right.\right\rbrace \subseteq \mathcal{D}
\end{align}
\end{subequations}
It can be noted that, without any loss of generality, the Lyapunov function can be scaled by $\gamma^{max}$, so that the ROA is given by,
\begin{subequations}\label{E:ROA2}
\begin{align}
\mathcal{R}_A:=&\left\lbrace x\in\mathbb{R}^n\left| {V}(x)\leq 1\right.\right\rbrace\\
\text{where,}~{V}(x)=& ~{\tilde{V}(x)}/{\gamma^{max}}
\end{align}
\end{subequations}
Henceforth, for simplicity, we would assume, without any serious loss of generality, that the ROA is estimated to be sub-level set of ${V}(x)=1$.

\subsection{Sum-of-Squares and Putinar's Positivestellensatz}
\label{S:SOSmethod}
 While the Theorem~\ref{T:Lyap} gives a sufficient condition for stability, it is still not a trivial task to find a suitable function $V\left(x\right)$ that satisfies the conditions of stability, even when the origin is a stable equilbrium point. Relatively recent studies have explored how sum-of-squares based optimization techniques can be utilized in finding Lyapunov functions by restricting the search space to sum-of-square polynomials \cite{Wloszek:2003,Parrilo:2000,Tan:2006,Anghel:2013}. Let us denote $\mathcal{R}_n$ as the set of all polynomials in $x\in\mathbb{R}^n$. Then,
\begin{definition}
A multivariate polynomial $p(x) \in \mathcal{R}_n$ is a sum-of-squares (SOS) if there exist some polynomial functions $h_i(x), i = 1\ldots r$ such that 
$p(x) = \sum_{i=1}^r h_i^2(x)$,
and the set of all such SOS polynomials is denoted by
\begin{align}
\Sigma_{n} &:= \left\lbrace p\left(x\right)\in\mathcal{R}_n\left| ~p \text{ is SOS}\right.\right\rbrace
\end{align}
\end{definition}
Given a polynomial $p\in\mathcal{R}_n$, checking if it is SOS is a semi-definite problem which can be solved with a MATLAB$^\text{\textregistered}$ toolbox SOSTOOLS \cite{sostools13,Antonis:2005a} along with a semidefinite programming solver such as SeDuMi \cite{Sturm:1999}.

SOS technique can be used to find a polynomial Lyapunov function $V\left(x\right):\mathbb{R}^n\rightarrow \mathbb{R}$, with $V(\mathbf{0}_{n\times 1})=0$, which satisfies the following SOS conditions \cite{sostools13,Antonis:2002,Wloszek:2003,Wloszek:2005,Antonis:2005,Antonis:2005a,Antonis:2005b },
\begin{subequations}\label{E:Lyap_SOS}
\begin{align}
 V(x) - \phi_1(x)   &\in \Sigma_n , \forall x \in  \mathcal{D}  \\
 -\dot{V}(x)    - \phi_2(x)   &    \in \Sigma_n , \forall x \in  \mathcal{D}
\end{align}
\end{subequations}
for some domain $\mathcal{D}$ {around the origin and positive definite functions $\phi_1(x),~\phi_2(x)$}. 
Often it is convenient to choose $\mathcal{D}:=\left\lbrace x\in\mathbb{R}^n\left\vert p(x)<\beta\right.\right\rbrace$, for some positive definite function $p(x)$ and $\beta>0$.

An important result from algebraic geometry called Putinar's Positivstellensatz theorem \cite{Putinar:1993,Lasserre:2009} helps in translating the SOS conditions into SOS feasibility problems. Before stating the theorem, let us define:
\begin{definition}\label{D:quad mod}
Given $g_j\in\mathcal{R}_n$, for $j=1,2,\dots,m$, the quadratic module generated by $g_j$'s is 
$\mathcal{M}(g_1,g_2,\dots,g_m):=\left\lbrace \sigma_0 + \sum_j\sigma_jg_j\left\vert \sigma_0,\sigma_j\in\Sigma_n,\forall j\right.\right\rbrace$
\end{definition}
Then the Putinar's Positivestellensatz theorem states,
{\begin{theorem}\label{T:Putinar}
Let $\mathcal{K}= \left\lbrace x\in\mathbb{R}^n\left\vert g_1(x) \geq 0, \dots , g_m(x)\geq 0\right.\right\rbrace$ be a
compact set. Suppose there exists $u(x)\in\mathcal{R}_n$ such that 
\begin{subequations}\label{E:Putinar}
\begin{align}
& u(x)\in\mathcal{M}(g_1, g_2,\dots , g_m),\\
\text{and,}~&\left\lbrace x\in\mathbb{R}^n\left\vert u(x)\geq 0\right.\right\rbrace~\text{is compact.}
\end{align}
\end{subequations}
If $p(x)$ is positive on $\mathcal{K}$, then $p(x)\in \mathcal{M}(g_1, g_2,\dots , g_m)$.
\end{theorem}}
Note: often in this work, {for the $g_i$'s used,} the constraints (\ref{E:Putinar}) {would be redundant, i.e. the existence of $u(x)$ would be guaranteed \cite{Lasserre:2009}}\footnote{Simplicity in formulating an SOS problem motivated us to use Putinar's version of Positivstellensatz over other, more general, versions of the Positivstellensatz theorem \cite{Lasserre:2009}.}.

\section{{PROBLEM OUTLINE}}\label{S:problem}
 {L}{et} us assume that the dynamical system in (\ref{E:f}) is in polynomial form, i.e. $f$ is a vector of $n$ polynomials\footnote{If the dynamics is not in polynomial form, it has to be recasted into a polynomial form, with possible additions of equality constraints \cite{Antonis:2002, Antonis:2005,Anghel:2013,Anghel:2013b}, a case not considered in this work.}. 
 Given the full dynamical system (\ref{E:f}), we can decompose it into subsystems
\begin{subequations}\label{E:fi}
\begin{align}
\forall i =1,2,\dots&,m,\nonumber\\
&\dot{x}_i = f_i(x_i) + g_i(x), \quad x_i\in\mathbb{R}^{n_i},\\
&f_i(\mathbf{0}_{n_i\times 1})=\mathbf{0}_{n_i\times 1},~g_i(\mathbf{0}_{n\times 1})=\mathbf{0}_{n_i\times 1}\\
& x = \left( x_1^T ,x_2^T, \dots, x_m^T\right)^T \in \mathbb{R}^{n}\\
& n=\sum_{i=1}^m n_i\, , ~x_i\cap x_j = \emptyset \label{E:non-overlap}
\end{align}
\end{subequations}
In the decomposed system description, the $f_i$'s denote the isolated subsystem dynamics, and $g_i$'s are the interactions from the neighbors \cite{Anghel:2013}. 
Let us also denote by
\begin{align}\label{E:Ni}
\mathcal{N}_i := \left\lbrace i\right\rbrace\cup\left\lbrace j\left\vert ~\exists x_i,x_j, ~\text{s.t.}~g_{ij}\left(x_i,x_j\right)\neq 0 \right.\right\rbrace,~\forall i,
\end{align}
the set of neighbors (including the subsystem itself). We assume that the isolated subsystems are individually (locally) stable, and  there exist Lyapunov functions for each of the isolated subsystems.
The goal is to develop a framework for the stability analysis of the full interconnected system by using the local subsystem Lyapunov functions and considering the neighbor interactions. 

Given a decomposition (\ref{E:fi}), the next step is to find polynomial Lyapunov functions $V_i(x_i):\mathbb{R}^{{n}_i}\rightarrow \mathbb{R}$, such that \cite{Wloszek:2003},
\begin{subequations}\label{E:initV}
\begin{align}
\forall x_i \in\mathcal{D}_i,&\quad V_i(x_i) - \phi_{i1}(x_i)\in \Sigma_{{n}_i},\label{E:Via}\\
\text{and,}&\quad -\nabla V_i(x_i)^T {f}_i(x_i) - \phi_{i2}(x_i)\in \Sigma_{{n}_i}\label{E:Vib}\\
\text{where,}&\quad \mathcal{D}_i:=\left\lbrace x_i\in\mathbb{R}^{n_i}\left| p_i(x_i)\leq \beta_i\right.\right\rbrace\notag
\end{align}
\end{subequations}
for some {positive definite functions $\phi_{i1}(x_i),~\phi_{i2}(x_i),~p_i(x_i)$} and positive scalars $\beta_i$. {Starting from an initial Lyapunov function candidate obtained using (\ref{E:initV}) and a corresponding estimate of the region of attraction, an iterative process called \textit{expanding interior algorithm}, \cite{Wloszek:2003,Anghel:2013}, is used to iteratively enlarge the estimate of the region of attraction by finding a better Lyapunov function at each step of the algorithm.} At the completion of this iterative step, the stability of each isolated subsystem (assuming no interaction) is quantified by its Lyapunov function $V_i(x_i)$, with an estimation of the boundary of the domain of attraction given by $\mathcal{R}_{A,i} = \left\lbrace x_i\in\mathbb{R}^{n_i}\left|V_i(x_i)\leq 1\right.\right\rbrace$.

The Lyapunov level-sets can be used to express the strength of a disturbance. The equilibrium point of the system at origin corresponds to the level set $V_i(\mathbf{0}_{n_i\times 1})=0,\forall i$.  If there is a disturbance from this equilibrium point, the states of the system would move to some point $x(0)$ away from the origin. This disturbed initial condition would result in positive level-sets $V_i(x_i(0))=\gamma_i^0\in\left(0,1\right]$ for some or all of the subsystems. 
A necessary and sufficient condition of asymptotic stability can then be translated into the condition
\begin{align}\label{E:cond_asymp}
\forall i, ~V_i(x_i(0))=\gamma_i^0\implies\forall i, ~\lim_{t\rightarrow +\infty}{V}_i(x_i(t))=0
\end{align}

In the rest of the article, we present SOS algorithms to test stability conditions and design local (subsystem-level) control laws to achieve asymptotic stability.

\section{{STABILITY UNDER INTERACTIONS}}\label{S:stability}
It is assumed that the isolated subsystems in (\ref{E:fi}) are all (locally) asymptotically stable, and there exist subsystem Lyapunov functions $V_i(x_i),\forall i$. {The estimated} region of attraction of the interconnected system under no interaction, $\mathcal{R}_A^0$, is given by the cross-product of the regions of attraction of the isolated subsystems, $\mathcal{R}_{A,i}$, which are defined as sub-unity-level sets of the corresponding (properly scaled) subsystem Lyapunov functions (as in (\ref{E:ROA})-(\ref{E:ROA2})), i.e.
\begin{subequations}\label{E:ROA_isol}
\begin{align}
\mathcal{R}_A^{0} :=&~\mathcal{R}_{A,1}\times \mathcal{R}_{A,2}\times \dots \times\mathcal{R}_{A,m} \\
\text{where},~ \mathcal{R}_{A,i} = &\left\lbrace x_i\in\mathbb{R}^{n_i}\left|V_i(x_i)\leq 1\right.\right\rbrace, ~\forall i
\end{align}
\end{subequations}
In presence of non-zero interactions, the resulting ROA would be different. If there exists a Lyapunov function for the interconnected system, the ROA for the whole system could be expressed as some sub-level set of that Lyapunov function. While it is very hard to obtain a scalar Lyapunov function for the full interconnected system, one could use vector Lyapunov function approach to obtain certification of stability in a scalable way.
%

In this present work, we choose not to impose any further restriction on the Lyapunov functions $V_i(x_i)$ than requiring that those are in polynomial forms, and concern ourselves with asymptotic stability. In general, it is difficult to test a necessary and sufficient condition of asymptotic stability, such as the one given in (\ref{E:cond_asymptotic}). Nevertheless, it is possible to derive sufficient conditions of asymptotic stability under certain scenarios. Let us now present a distributed iterative procedure which can be used to certify asymptotic stability in a domain defined by sub-level sets of the subsystem Lyapunov functions.

\subsection{Algorithmic Test of Asymptotic Stabiltiy}\label{S:asymptotic}
Before proceeding to explaining our algorithm, let us first note the following result:
\begin{lemma}\label{L:asymptotic}
Suppose, for all $i\in\left\lbrace 1,2,\dots,m\right\rbrace$, there exists a strictly monotonically decreasing sequence of scalars $\left\lbrace \epsilon_i^k\right\rbrace,k\in\left\lbrace 0,1,2,\dots\right\rbrace$, such that
\begin{subequations}\label{E:cond_asymptotic}
\begin{align}
\forall i,k,~\dot{V}_i(x)&:=\nabla{V}_i(x_i)^T\left(f_i(x_i)+g_i(x)\right)<0,~\forall x\in\mathcal{D}_i^k\\
\text{where,}~\mathcal{D}_i^k &:=\left\lbrace x\in\mathbb{R}^n\left\vert \begin{array}{c}\epsilon_i^{k+1}\leq V_i(x_i)\leq\epsilon_i^k,\\V_j(x_j)\leq\epsilon_j^k~\forall j\neq i\end{array}\right.\right\rbrace \label{E:Dk}
\end{align}
\end{subequations}
Then the system (\ref{E:f}) is asymptotically stable in the domain $\left\lbrace x\in\mathbb{R}^n\left| \bigcap_{i=1}^m V_i(x_i)\leq \epsilon_{i}^0\right.\right\rbrace$, if $\lim_{k\rightarrow +\infty} \epsilon_i^k= 0,\forall i$\footnote{{If the limit condition does not hold, we can only guarantee stability in the sense of Lyapunov \cite{Slotine:1991,Lyapunov:1892}}}.
\end{lemma}
\begin{proof}
Please refer to Appendix~\ref{A:proof}.
\end{proof}

Using the Lemma~\ref{L:asymptotic} we can devise a simple iterative SOS algorithm to certify whether or not a domain $\mathcal{D}$ defined by
\begin{align}\label{E:D}
\mathcal{D}:=\left\lbrace x\in\mathbb{R}^n\left| \bigcap_{i=1}^m V_i(x_i)\leq v_{oi}\right.\right\rbrace, 
\end{align}
for some scalars $v_{oi}\in\left(0,1\right],\forall i$, is a region of asymptotic stability for the system in (\ref{E:f}). It is to be noted that, using the Putinar's Positivstellensatz theorem (Theorem~\ref{T:Putinar}), the condition in (\ref{E:cond_asymptotic}) essentially translates into equivalent SOS feasibility conditions
\begin{align}\label{E:cond_asymptotic_SOS}
&\forall i,k,~\exists \sigma_{i0}^k,\sigma_{ij}^k\in\Sigma_{\bar{n}_i},~\text{s.t.}\\
&-\nabla{V}_i^T\left(f_i+g_i\right) -\sigma_{i0}^k\left(V_i-\epsilon_i^{k+1}\right) - \sum_{j\in\mathcal{N}_i}\sigma_{ij}^k\left(\epsilon_j^k-V_j\right)\in\Sigma_{\bar{n}_i}\notag\\
&\text{where,}~\bar{n}_i = \sum_{j\in\mathcal{N}_i}n_j.\notag
\end{align}
The algorithmic steps to ascertain asymptotic stability are as outlined below:
\begin{enumerate}
\item We initialize $\epsilon_i^0=v_{oi},\forall i\in\left\lbrace 1,2,\dots,m\right \rbrace$, and choose a sufficiently small $\bar{\epsilon}\in\mathbb{R}^+$.

\item\label{I:iteration} At the start of the $k$-th iteration loop, we assume to {know} the scalars $\left\lbrace \epsilon_i^0,\epsilon_i^1,\dots,\epsilon_i^k\right\rbrace,\forall i$, and our aim is to compute the scalars {$\epsilon_i^{k+1},~\forall i$} such that (\ref{E:cond_asymptotic_SOS}) holds. Essentially we want to solve the optimization problem,
\begin{subequations}\label{E:max_epsilon}
\begin{align}
\forall i,~ &\min_{\sigma_{i0}^k,\sigma_{ij}^k,\epsilon_i^{k+1}} ~\epsilon_i^{k+1}\\
&\text{s.t.},~ \text{the condition in (\ref{E:cond_asymptotic_SOS}) holds.}
\end{align}
\end{subequations}
This is solved by performing a bisection search for minimum $\epsilon_i^{k+1}$ over the range $\left[0,\epsilon_i^k\right]$. 

If (\ref{E:max_epsilon}) is infeasible at $0$-th iteration\footnote{Because $V_i$'s, $f_i$'s and $g_i$'s are polynomial, if (\ref{E:max_epsilon}) is feasible at $0$-th iteration for all $i$, then it is also feasible for all subsequent iterations.} for any $i\in\left\lbrace 1,2,\dots,m\right\rbrace$, we conclude that the system cannot be guaranteed to be asymptotically stable in $\mathcal{D}$, and abort the iteration. Otherwise we move on to step \ref{I:conclusion}.

\item\label{I:conclusion} If $\left(\epsilon_i^k-\epsilon_i^{k+1}\right)\geq\bar{\epsilon},\forall i$, we continue from step \ref{I:iteration} for the ($k$+1)-th iteration loop. Otherwise we stop the iteration deciding that the limits of the sequences $\left\lbrace\epsilon_i^k\right\rbrace,\forall i,$ have been attained. Further, if the {limits are all zero, we certify asymptotic stable in $\mathcal{D}$}.
\end{enumerate}

\subsection{Remarks}\label{S:asymptotic_remarks}
The algorithm presented in Sec.~\ref{S:asymptotic} describes how one can determine asymptotic stability of an interconnected system in a domain $\mathcal{D}$ defined by the subsystem sub-level sets. This test can be performed locally, and in a parallel way, at each subsystem level. The Lyapunov functions $V_i$'s are to be found before the start of the analysis, and communicated to the neighboring subsystems. Then during each analysis, it is assumed that the neighboring subsystems can communicate with each other the computed sequences $\left\lbrace \epsilon_i^k\right\rbrace$ {in} real-time. With the help of {the} stored Lyapunov functions, and the updated $\left\lbrace \epsilon_i^k\right\rbrace$ of the neighbors, each subsystem will continue the iterative process outlined in Sec.~\ref{S:asymptotic}. Since only the neighbor information is required, this algorithm is reasonably scalable with respect to the size of the full interconnected system. Moreover, the algorithm motivates the design of a distributed control strategy that can ascertain asymptotic stability.

\section{{DECENTRALIZED CONTROL}}\label{S:control}
In this section, we discuss design of a \textit{local} and \textit{minimal} control strategy such that the system in (\ref{E:f}) is asymptotically stable in a domain $\mathcal{D}$ defined in (\ref{E:D}). We use the term \textit{minimal} to suggest that the control be applied only in certain regions, and not everywhere, in the state space, while by the term \textit{local} we suggest that the control be computable and implementable on a subsystem level.  

We envision the control to be computed by each subsystem at each iteration loop. At $k$-th iteration, $\forall k\in\left\lbrace 0,1,2,\dots\right\rbrace$, the  $i$-th subsystem, $\forall i\in\left\lbrace 1,2,\dots,m\right\rbrace$ performs the following tasks:
\begin{enumerate}
\item It identifies if it belongs to the following set 
\begin{align}
\mathcal{U}^k&:= \left\lbrace i\in\left\lbrace 1,\dots,m\right\rbrace\left\vert \begin{array}{c}\nabla{V}_i^T\left(f_i+g_i\right)\geq0, \\V_i=\epsilon_i^k,\\V_j\leq\epsilon_j^k~\forall j\in\mathcal{N}_i\setminus\{i\}\end{array}\right.\right\rbrace
\end{align}
which can be checked locally. If $i\notin\mathcal{U}^k$, control is not necessary and it sets $F_i^k\equiv\mathbf{0}_{n_i\times 1}$ and proceeds to task~\ref{I:control_dyn}. If, however, $i\in\mathcal{U}^k$, it proceeds to task~\ref{I:Fi} to compute a control law.

\item\label{I:Fi} If $i\in\mathcal{U}$, a polynomial state-feedback control law $F_i^k:\mathbb{R}^{n_i}\rightarrow\mathbb{R}^{n_i}$, $F_i^k(\mathbf{0}_{n_i\times 1} )= \mathbf{0}_{n_i\times 1}$, is {computed such that}
\begin{align}\label{E:Fi}
\left.\nabla{V}_i^T\!\left(f_i+g_i + F_i^k\right)\right\vert_{\left\lbrace V_i=\epsilon_i^k, ~V_j\leq\epsilon_j^k~\forall j\in\mathcal{N}_i\setminus\{i\}\right\rbrace}\!<\!0.
\end{align}
This produces the equivalent SOS condition,
\begin{align}\label{E:sos_K}
&-\!\!\nabla{V}_i^T\!\!\!\left(\!f_i\!+\!g_i\!+\!F_i^k\!\!\right)\!\! -\!\!\rho_{i}^k\!\!\left(\!\epsilon_i^{k}\!-\!V_i\!\!\right) \!\!-\!\! \!\!\sum_{j\in\mathcal{N}_i\setminus\{i\}}\!\!\!\!\sigma_{ij}^k\!\!\left(\!\epsilon_j^k\!-\!V_j\!\!\right)\!\!\in\!\!\Sigma_{\bar{n}_i}\notag\\
&\rho_{i}^k\in\mathcal{R}_{\bar{n}_i},~\sigma_{ij}^k\in\Sigma_{\bar{n}_i}\forall j\neq i,~\bar{n}_i = \sum_{j\in\mathcal{N}_i}n_j
\end{align}

\item\label{I:control_dyn} Finally it performs the search over minimum $\epsilon_i^{k+1}$, as in (\ref{E:max_epsilon}), with the un-controlled subsystem dynamics $\left(f_i+g_i\right)$ in the feasibility condition (\ref{E:cond_asymptotic_SOS}) is replaced by the controlled dynamics $\left(f_i+g_i+F_i^k\right)$.
\end{enumerate}

To summarize, each subsystem $i$ computes control laws $F_i^k:\mathbb{R}^{n_i}\rightarrow\mathbb{R}^{n_i}$, with $F_i^k(\mathbf{0}_{n_i\times 1} )= \mathbf{0}_{n_i\times 1}$, during each $k$-th iteration, so that the subsystem dynamics under control becomes:
\begin{align}
\forall i,~\forall k, ~&\forall x\in\mathcal{D}_i^k,\notag\\
\dot{x}_i &= \left\lbrace \begin{array}{ll}f_i(x_i) + g_i(x), ~&i\notin\mathcal{U}^k\\
			f_i(x_i) + g_i(x)+F_i^k(x_i), ~&i\in\mathcal{U}^k\end{array}\right.
\end{align}
where $\mathcal{D}_i^k$ were defined in (\ref{E:Dk}).

{\subsection{Remarks}
Often it is important to impose certain additional constraints on the possible control laws, such as bounds on the control effort. Although control bounds can be easily incorporated in the SOS formulation, we decide to keep that for future studies. We note, however, that since we apply controls $F_i^k$ only on certain subsystems $i\in\mathcal{U}^k$, and in certain domains $\mathcal{D}_i^k\subseteq\mathcal{D}$, the control effort would be reasonably bounded.}

\section{{RESULTS}}\label{S:results}
Let us describe the model of the interconnected system that we use here, and two examples to illustrate the applications of the stability analysis algorithm and control design.
 
 \subsection{Model Description}\label{S:model}
 {W}{e} will consider a network of nine Van der Pol oscillators \cite{van1926}, as shown in Fig.~\ref{F:osc_net9}. 
 \begin{figure}[thpb]
      \centering
	\includegraphics[width=3.5in]{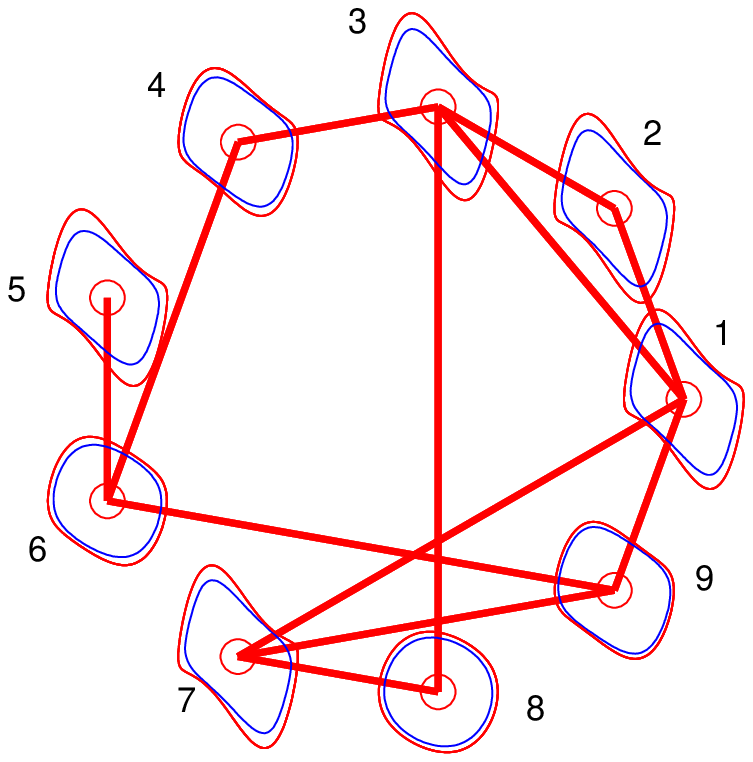} 
	      \caption{A network of nine Van der Pol oscillators along with their isolated regions of attraction.}
      \label{F:osc_net9}
   \end{figure}  
    \begin{figure}[thpb]
      \centering
	\includegraphics[width=3in]{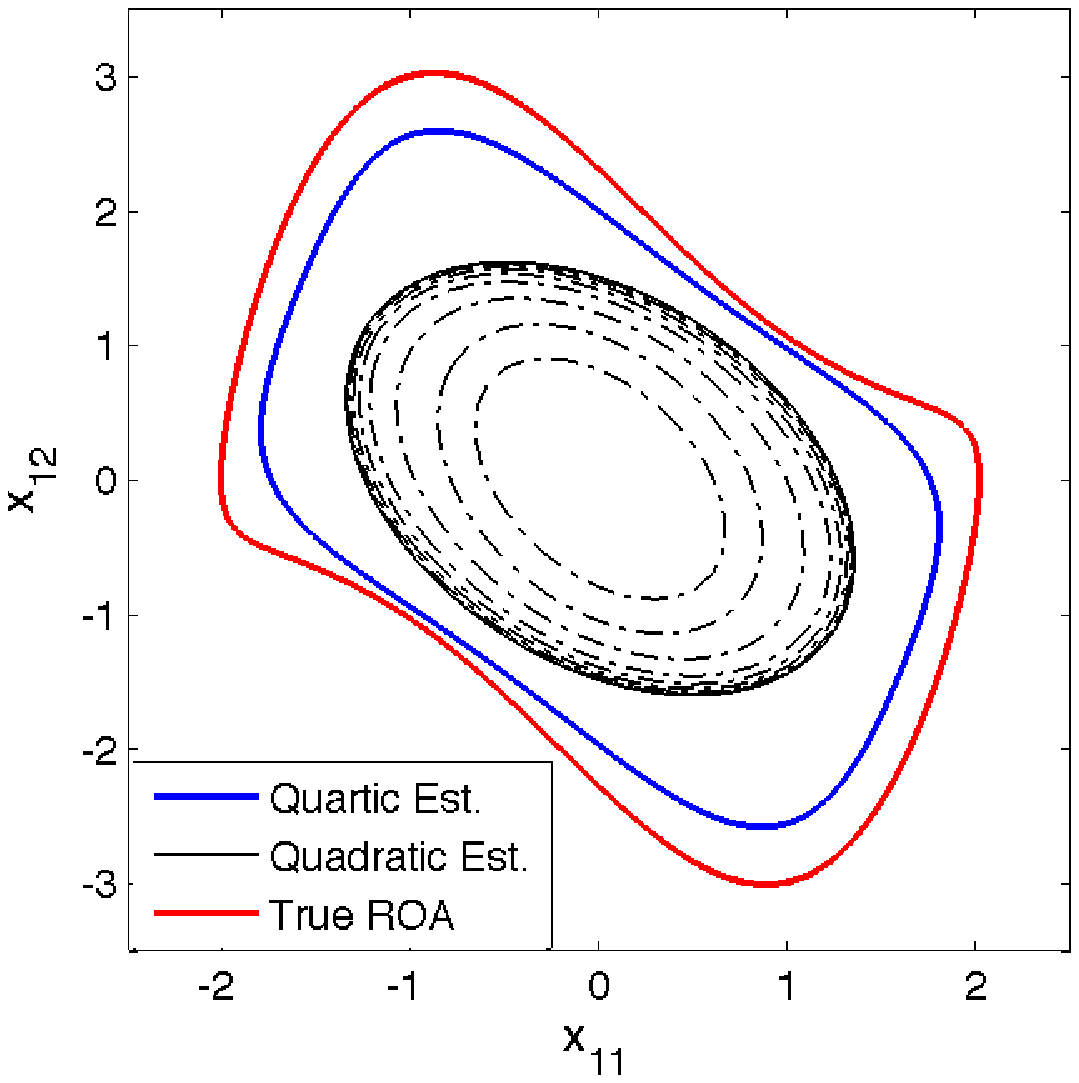}
      \caption{Comparison of estimated ROAs with true ROA for oscillator $1$.}
      \label{F:ROAcompare}
   \end{figure}   
The dynamics of each oscillator, in presence of neighbor interactions, is represented by  
 \begin{subequations}
\begin{align}
\forall j\in&\left\lbrace 1,2,\dots,9\right\rbrace,\notag\\
	&\dot{x}_{j1}= x_{j2} \\
	&\dot{x}_{j2}=\mu_jx_{j2}\left(1-x_{j1}^2\right) - x_{j1} +x_{j1}\sum_{k\neq j}\zeta_{jk}x_{k2}
\end{align}
\end{subequations}
where $\mu_j$'s are chosen randomly from $\left(-2,~0\right)$ and the coefficients, $\zeta_{jk}$, of the interaction terms are chosen randomly from $\left(-0.2,~0.2\right)$. It is to be noted, that the interactions need not be symmetric, i.e. in general, $\zeta_{ik}\neq\zeta_{ki}$. Additionally, $\zeta_{jk}=0$ if oscillator $k$ is not a neighbor of oscillator $j$. Such choice of $\mu_j$'s ensure that the oscillators themselves are stable, with corresponding regions of attraction as shown in Fig.~\ref{F:osc_net9}. The regions drawn in `red' around each oscillator shows its true ROA, while the region in `blue' shows an estimate of the ROA as sub-unity-level set of its polynomial Lyapunov function, as in (\ref{E:ROA2}). In Fig.~\ref{F:ROAcompare} we compare with the true ROA the estimates obtained using a quartic Lyapunov function and a quadratic one. Also a sequence of estimates using the quadratic Lyapunov function are shown in `dotted black' lines, which show how the `expanding interior' algorithm iteratively expands the estimate of the ROA. Clearly, the final estimate improves as the degree of the polynomial Lyapunov function increases, but for this work we choose to stick to quadratic Lyapunov functions.

We decompose the system into $7$ subsystems\footnote{This decomposition is arbitrary. \cite{Antonis:2012} presents a method of decomposition in weakly interacting subsystem, which however requires symmetric interactions.}, by grouping together oscillators $\left\lbrace2,3\right\rbrace$ and $\left\lbrace 5,6\right\rbrace$, as shown below
\begin{align}
\begin{array}{ll}
S_1:\left\lbrace osc~1\right\rbrace; & \mathcal{N}_1:\left\lbrace S_2,S_5,S_7\right\rbrace\\
S_2:\left\lbrace osc~2,~osc~3\right\rbrace; & \mathcal{N}_2:\left\lbrace S_1,S_3,S_6\right\rbrace\\
S_3:\left\lbrace osc~4\right\rbrace; & \mathcal{N}_3:\left\lbrace S_2,S_4\right\rbrace\\
S_4:\left\lbrace osc~5,~osc~6\right\rbrace; & \mathcal{N}_4:\left\lbrace S_3,S_7\right\rbrace\\
S_5:\left\lbrace osc~7\right\rbrace; & \mathcal{N}_5:\left\lbrace S_1,S_6,S_7\right\rbrace\\
S_6:\left\lbrace osc~8\right\rbrace; & \mathcal{N}_6:\left\lbrace S_2,S_5\right\rbrace\\
S_7:\left\lbrace osc~9\right\rbrace; & \mathcal{N}_7:\left\lbrace S_1,S_4,S_5\right\rbrace
\end{array}
\end{align}
Then we can write the subsystem dynamics, along with the neighbor interactions, in the form of (\ref{E:fi}). As an example, the states and the dynamics of $S_2$ are shown below,
\begin{align}
&\dot{x}_2 = f_2(x_2)+g_2(x);\quad x_2=\left(x_{21},~x_{22},~x_{31},~x_{32}\right)^T, \notag\\
&f_2=\left[\begin{array}{c}x_{22}\\ 0.12x_{21}x_{32}-x_{21}-0.41x_{22}\left(1-x_{21}^2\right)\\ x_{32}\\ 0.04x_{31}x_{22}-x_{31}-1.44x_{32}\left(1-x_{31}^2\right)\end{array}\right], \notag\\
&g_2=\left[\begin{array}{c} 0\\ -0.07x_{12}x_{21}\\0 \\ 0.01x_{12}x_{31} + 0.06x_{42}x_{31} + 0.1x_{82}x_{31}\end{array}\right]
\end{align}

Any randomly picked initial condition, $x(0)\in\mathcal{R}_A^0$, where $\mathcal{R}_A^0$ is defined in (\ref{E:ROA_isol}), can be mapped into corresponding subsystem Lyapunov function level sets, $\gamma_i^0=V_i(x_i(0)),\forall i$. Then, by choosing $v_{oi}=\gamma_i^0$, we apply the iterative stability analysis algorithm to determine whether or not the domain $\mathcal{D}$ in (\ref{E:D}) is a region of asymptotic stability, and if not, compute the necessary control by (\ref{E:Fi}).

\subsection{{Example: Certifiably Stable without Control}}
\begin{figure*}[thpb]
\centering
\subfigure[]{
\includegraphics[width=3in]{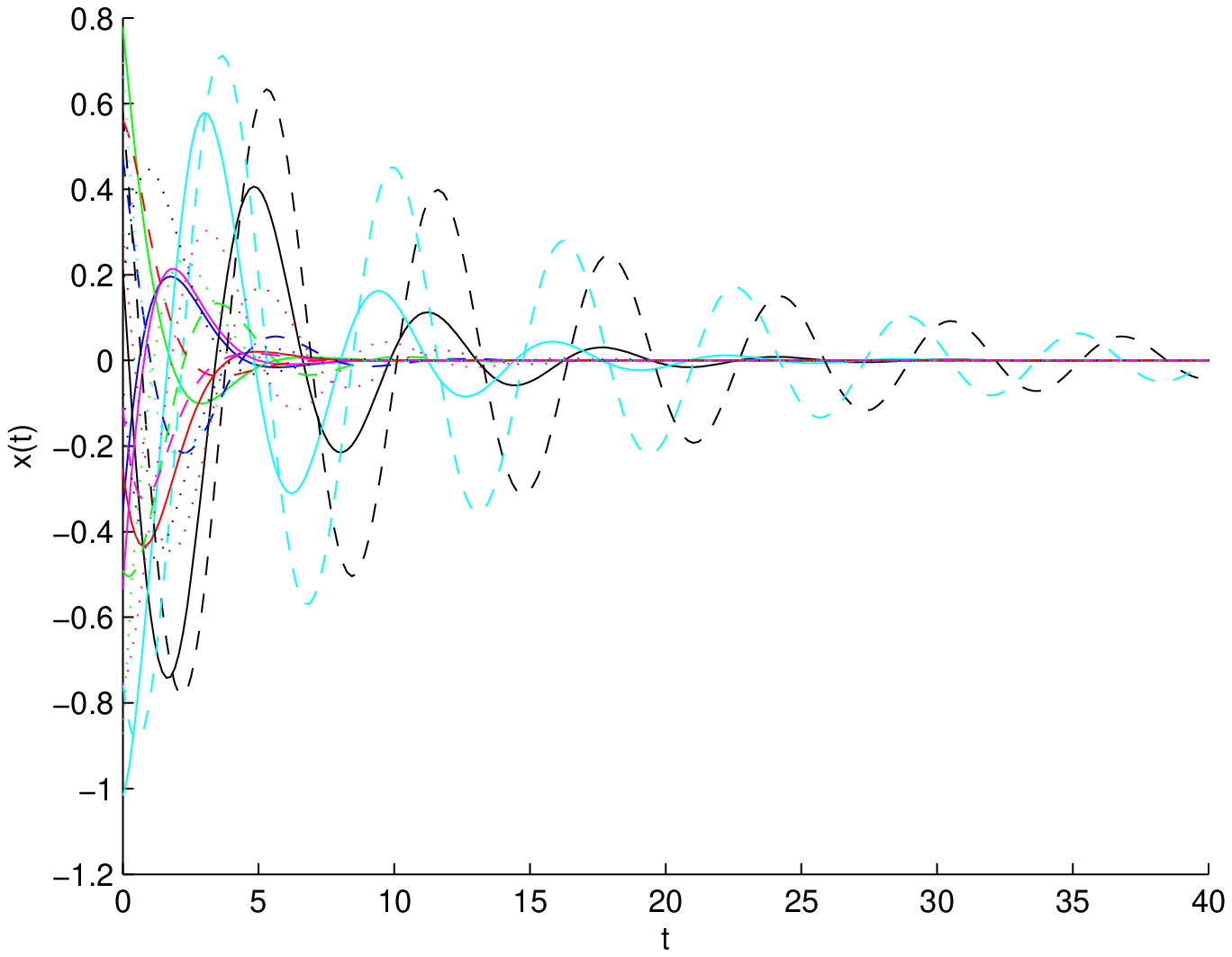}\label{F:dynamics}
}
\subfigure[]{
\includegraphics[width=3in]{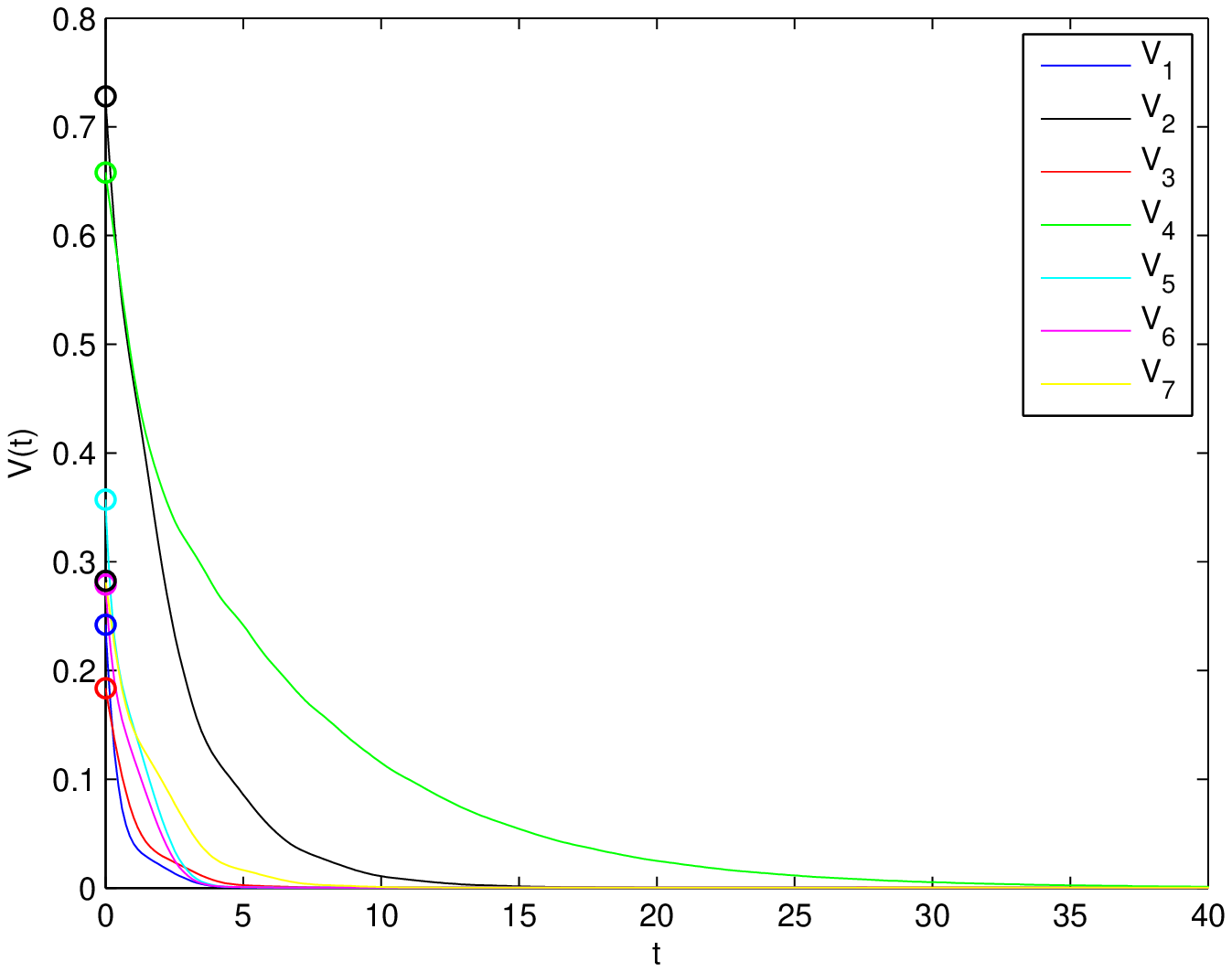}\label{F:dynamics_V}
}\caption[Optional caption for list of figures]{Evolution of states and the subsystem Lyapunov functions for a certifiably stable initial condition.}
\label{F:cert_stable}
\end{figure*}
   In Fig.~\ref{F:dynamics}, the evolutions of all the states to an asymptotically stable initial condition is shown. The subsystem Lyapunov functions, shown in Fig.~\ref{F:dynamics_V}, monotonically decrease to zero starting from the initial level sets: 
   \begin{align}
   \gamma_1^0=0.242,   ~\gamma_2^0= 0.728,  & ~\gamma_3^0= 0.184, ~ \gamma_4^0=  0.658, \notag\\
   \gamma_5^0=   0.357,  ~\gamma_6^0&=  0.279, ~ \gamma_7^0=  0.283
   \end{align} 
     Then setting $v_{oi}=\gamma_i^0$, and choosing $\bar{\epsilon}=0.001$, we run the iterative stability algorithm which produces the results in Table~\ref{Tab:stability}. At the end of the $2$nd iteration, all the $\epsilon_i^k$'s are zero, which certifies asymptotic stability of the full interconnected system.
   \begin{table}[h]
\caption{Iteration results for a certifiably stable case}
\label{Tab:stability}
\begin{center}
\begin{tabular}{|c|c|c|c|c|c|c|c|}
\hline
$k$ & $\epsilon_1^k$ & $\epsilon_2^k$ & $\epsilon_3^k$ & $\epsilon_4^k$ & $\epsilon_5^k$ & $\epsilon_6^k$ & $\epsilon_7^k$\\
\hline
0 & 0.242 & 0.728 & 0.184 & 0.658 & 0.357 & 0.279 & 0.283\\
\hline
1 & 0.024 &   0.031 &   0.147 &   0.188 &   0.004 &   0.010 &   0.127\\
\hline
2 & 0.000  &  0.000 &   0.000  &  0.000  &  0.000 &   0.000 &   0.000\\
\hline
\end{tabular}
\end{center}
\end{table}

\subsection{Example: Certifiably Stable under Control}
\begin{figure*}[thpb]
\centering
\subfigure[]{
\includegraphics[width=3in]{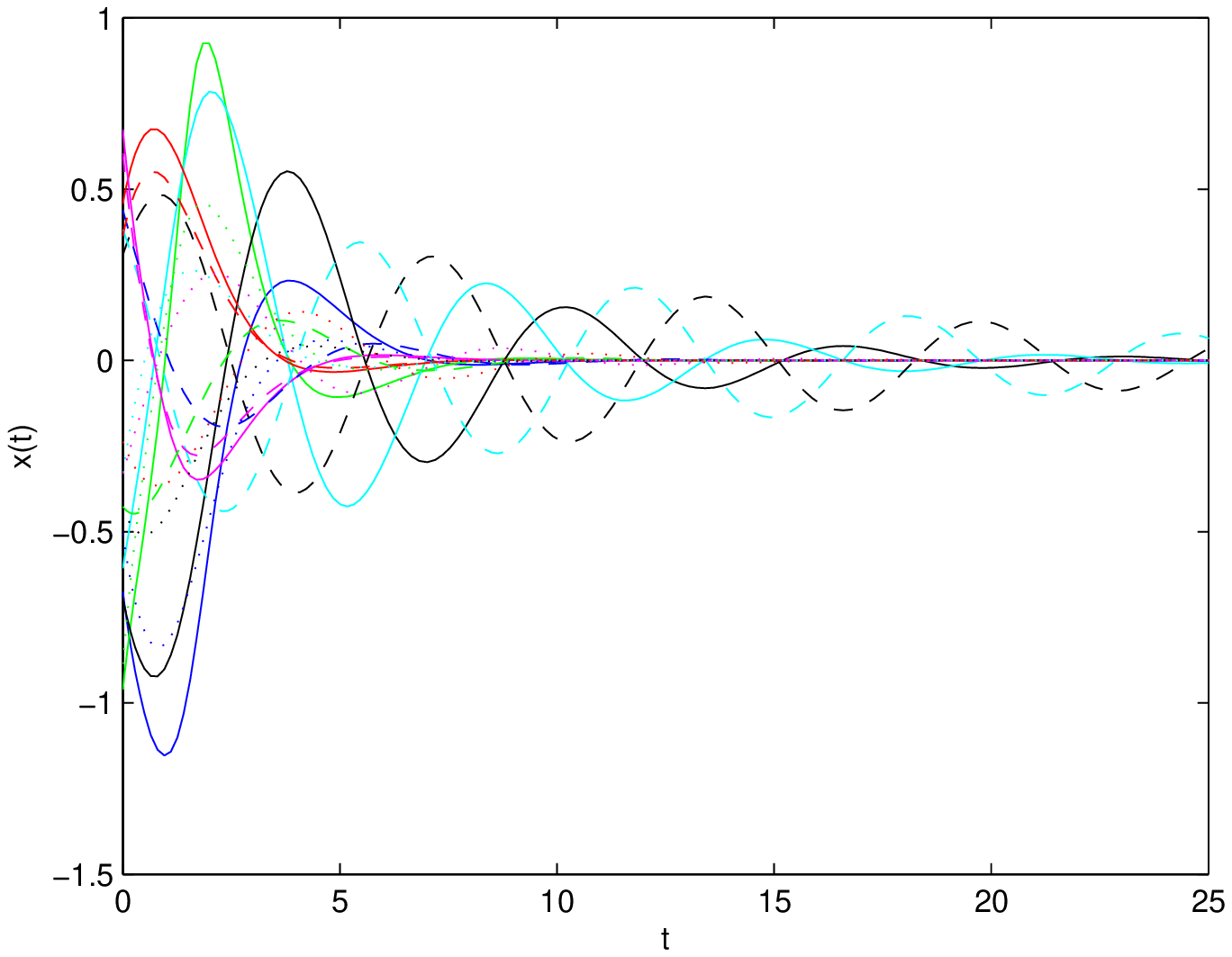}\label{F:dynamics_unstab}
}
\subfigure[]{
\includegraphics[width=3in]{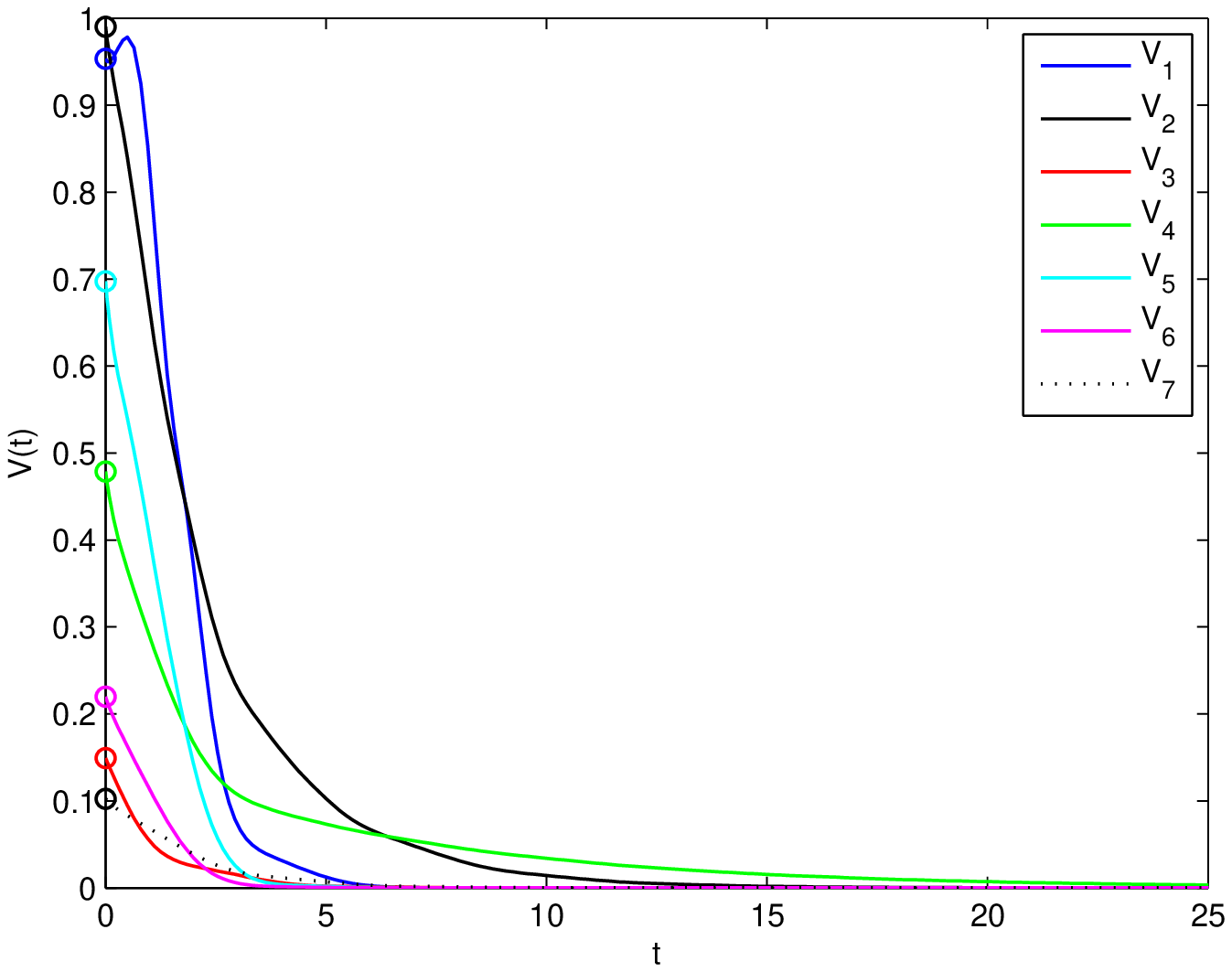}\label{F:dynamics_unstab_V}
}\caption[Optional caption for list of figures]{Evolution of states and the subsystem Lyapunov functions for a non-certifiably-stable initial condition.}
\label{F:uncert_stable}
\end{figure*}
Let us now present one example where the iterative algorithm fails to guarantee stability, and control is applied. Fig.~\ref{F:uncert_stable} shows a stable initial condition, but the algorithm fails to certify stability for it. In Fig.~\ref{F:dynamics_unstab_V}, $V_1(t)$ is seen to increase initially before starting to monotonically decrease. When we apply the stability analysis algorithm to the initial level sets:
   \begin{align}
   \gamma_1^0=0.953,   ~\gamma_2^0= 0.990,  & ~\gamma_3^0= 0.149, ~ \gamma_4^0=  0.479, \notag\\
   \gamma_5^0=   0.697,  ~\gamma_6^0&=  0.220, ~ \gamma_7^0=  0.103
   \end{align} 
the iteration fails at the first iteration because the algorithm cannot find feasible $\epsilon_i^1$'s for $i=1,2,5$ (as shown in Table~\ref{Tab:instability}).
\begin{table}[h]
\caption{Iteration fails to certify stability}
\label{Tab:instability}
\begin{center}
\begin{tabular}{|c|c|c|c|c|c|c|c|}
\hline
$k$ & $\epsilon_1^k$ & $\epsilon_2^k$ & $\epsilon_3^k$ & $\epsilon_4^k$ & $\epsilon_5^k$ & $\epsilon_6^k$ & $\epsilon_7^k$\\
\hline
0 & 0.953  &  0.990  &  0.149  &  0.479   & 0.697  & 0.220 &   0.103\\
\hline
1 & $\times$  &  $\times$ &   0.1010  &  0.0056  &  $\times$  &  0.0198   & 0.0380\\
\hline
\end{tabular}
\end{center}
\end{table}

Consequently decentralized control is activated for subsystems $S_1,S_2$ and $S_5$. This results in certifiable asymptotic stability, with new $\epsilon_i^k$'s shown in Table~\ref{Tab:stability_control} where the `$^*$' denotes presence of controllers\footnote{In this example, we chose to seek linear controllers which proved to be sufficient.} which are documented in (\ref{E:controllers}).
\begin{table}[h]
\caption{Iteration certifies stability under control}
\label{Tab:stability_control}
\begin{center}
\begin{tabular}{|c|c|c|c|c|c|c|c|}
\hline
$k$ & $\epsilon_1^k$ & $\epsilon_2^k$ & $\epsilon_3^k$ & $\epsilon_4^k$ & $\epsilon_5^k$ & $\epsilon_6^k$ & $\epsilon_7^k$\\
\hline
0 & 0.953  &  0.990  &  0.149  &  0.479   & 0.697  & 0.220 &   0.103\\
\hline
1 & 0.004$^*$  &  0.004$^*$ &   0.1010  &  0.0056  &  0.003$^*$  &  0.0198   & 0.0380\\
\hline
2 & 0.000 & 0.000 & 0.000 & 0.000 & 0.000 & 0.000 & 0.000\\
\hline 
\end{tabular}
\end{center}
\end{table}
\begin{align}\label{E:controllers}
F_1^1 &\!\!= \!\left[0, - 0.569x_{11} - 2.271x_{12}\right]^T,\notag\\
F_2^1 &\!\!=\!\left[0, - \!1.237x_{22}\!-\!0.149x_{21},~\!\!0\,,  - 0.285x_{31} \!\!-\!\! 1.368x_{32}\right]^T,\notag\\
F_5^1 &\!\!=\! \left[0, - 0.504x_{71} - 1.539x_{72}\right]^T
\end{align}
It is to be noted that we decide to apply control only on the dynamics equations of the states $x_{12},x_{22},x_{32}$ and $x_{72}$. Fig.~\ref{F:stable_control} shows that under the action of the controllers in (\ref{E:controllers}), all the subsystem Lyapunov functions decrease monotonically to zero.
\begin{figure*}[thpb]
\centering
\subfigure[]{
\includegraphics[width=3in]{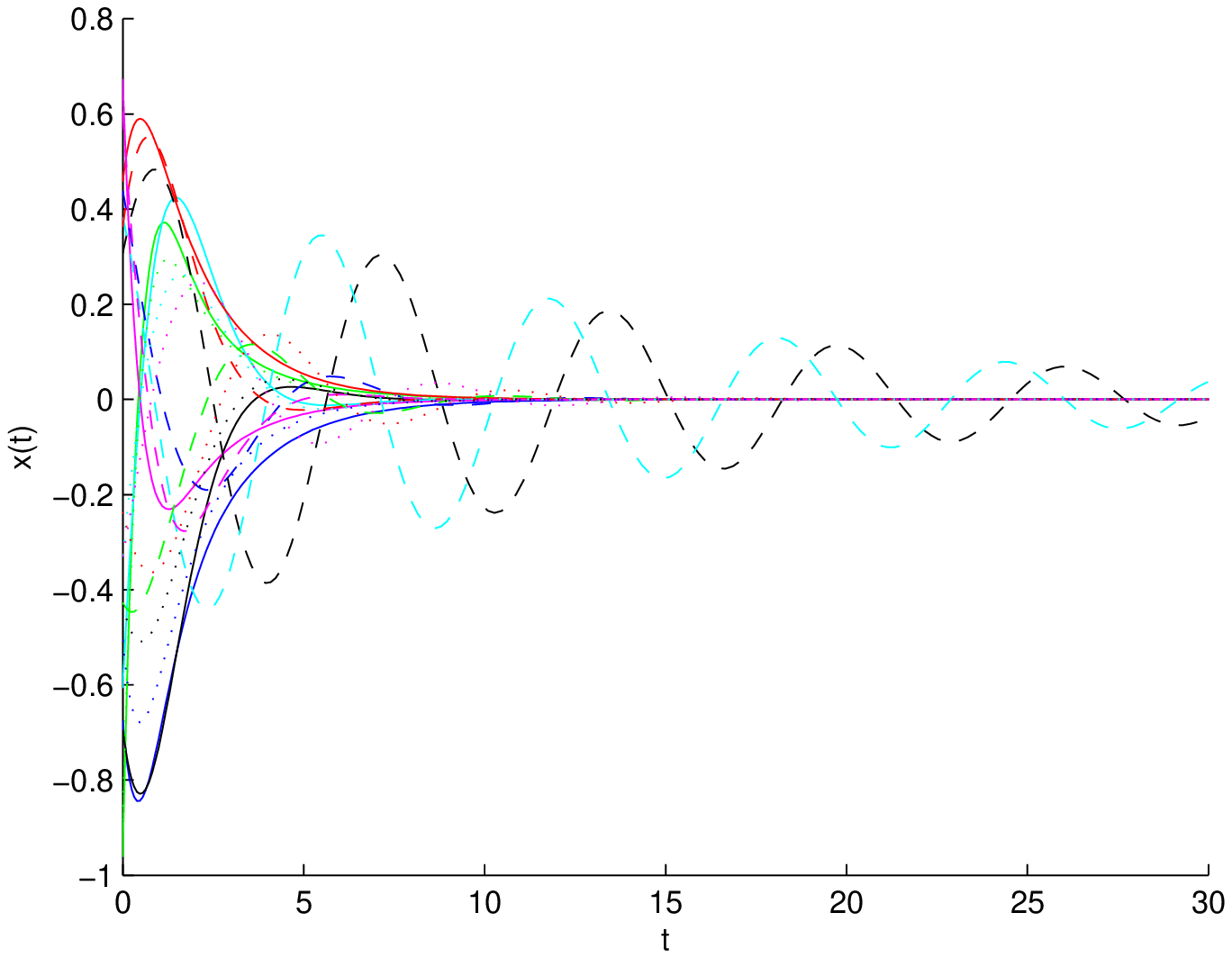}\label{F:dynamics_control}
}
\subfigure[]{
\includegraphics[width=3in]{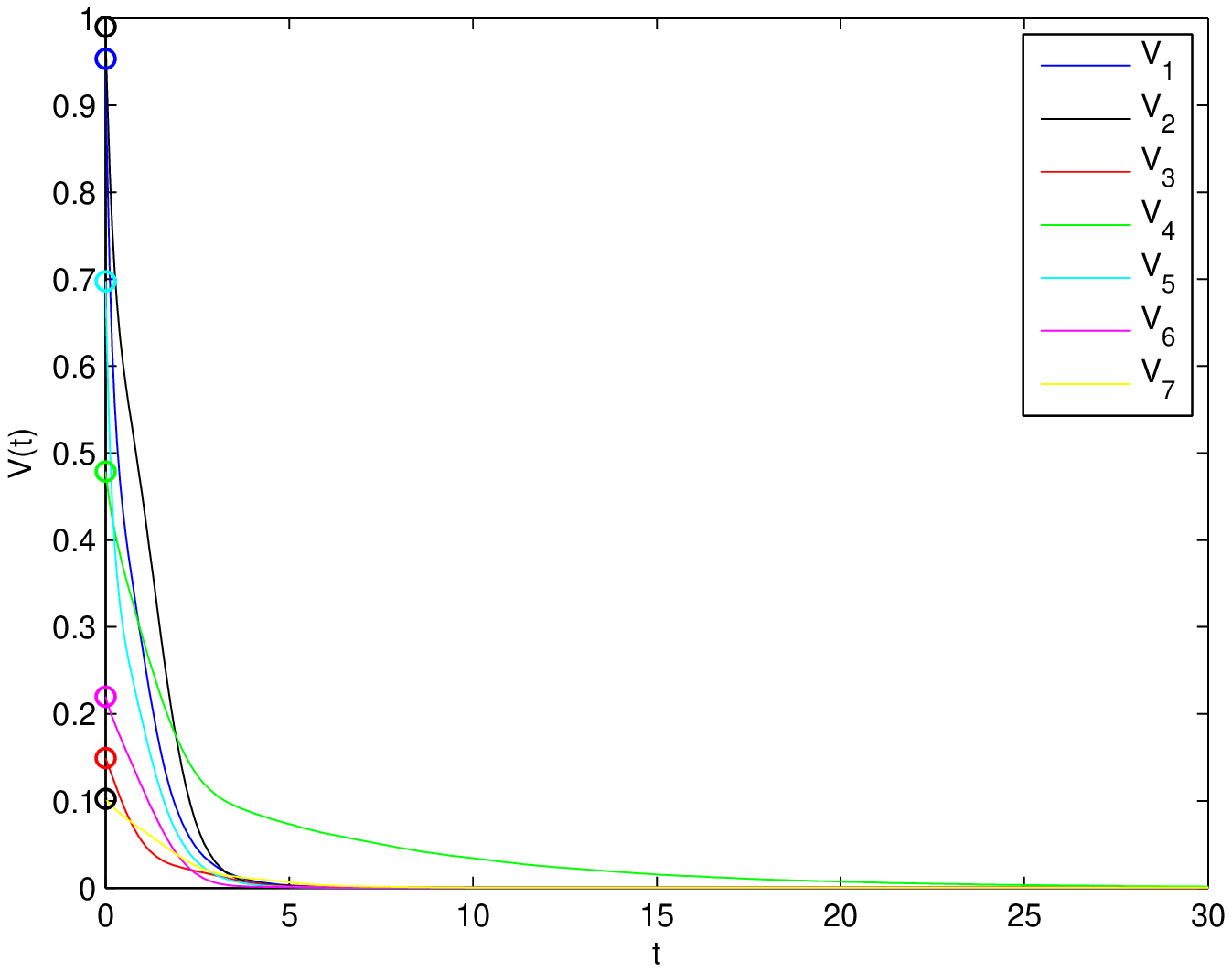}\label{F:dynamics_control_V}
}\caption[]{Evolution of states and the subsystem Lyapunov functions, under decentralized control applied at $S_1,S_2$ and $S_5$, that certifies stability.}
\label{F:stable_control}
\end{figure*}

\section{CONCLUSIONS}\label{S:conclusion}
{In this work, we present an algorithmic approach to certify asymptotic stability of an interconnected system whose dynamics can be expressed in polynomial form. We also propose the design of decentralized control laws when such a certification is not possible. The approach presented here is parallel and scalable. Similar method can also be applicable to complex real world systems, such as the power system. While power system dynamics are non-polynomial, if those are transformed into polynomial forms, by introducing additional equality constraints \cite{Anghel:2013}, the methods developed in this article can be applied. Future work need to address the issues such as including bounds on the control effort, and relaxing the requirement of monotonic decrease of subsystem Lyapunov functions along the flow.}

\appendix

\subsection{Proof of Lemma~\ref{L:asymptotic}}\label{A:proof}

{We note that since $\lim_{k\rightarrow +\infty} \epsilon_i^k= 0,\forall i$, 
\begin{align}\label{E:proof_delta}
\forall \delta\in\left(0,\min_i\epsilon_i^0\right], ~\exists K, ~\text{s.t.}~\epsilon_i^k<\delta~\forall k>K,\forall i.
\end{align}
Let us assume, without any loss of generality, that 
\begin{align}
\exists t_0\geq 0, ~\text{s.t.}~x(t_0)\in\left\lbrace x\in\mathbb{R}^n\left| \bigcap_{i=1}^m V_i(x_i)\leq \epsilon_{i}^0\right.\right\rbrace
\end{align}
Then,
\begin{subequations}
\begin{align}
\forall i, \quad& V_i(t) = \epsilon_i^0 + \int_{t_0}^t \dot{V}_i(\tau)d\tau,\quad \forall t\geq t_0\\
\implies & \exists~\! t_i^1 < t_0 + \left(\epsilon_i^1-\epsilon_i^0\right)/\bar{r}_i^1,\quad \bar{r}_i^1:=\sup_{x\in\mathcal{D}_i^1} \dot{V}_i(x)<0\\
	& \text{s.t.}~V_i(t)<\epsilon_i^1,~\forall t\geq t_i^1
\end{align}
\end{subequations}
Hence we can argue that,
\begin{align}
V_i(t)\leq\epsilon_i^0,&~\forall t\geq t_0,\forall i\notag\\
\implies &\exists~\! t^1:=\max_i t_i^1,~\text{s.t.}~V_i(t)\leq \epsilon_i^1,\forall t\geq t^1,\forall i
\end{align}
Following similar arguments it is easy to show that,
\begin{align}\label{E:proof_tk}
V_i(t)\leq\epsilon_i^0 &,~\forall t\geq t_0,~\forall i\notag\\
\implies &\forall k, ~\exists~\! t^k\geq t_0,~\text{s.t.}~V_i(t)\leq \epsilon_i^k,~\forall t\geq t^k,~\forall i
\end{align}
Finally combining (\ref{E:proof_delta}) and (\ref{E:proof_tk}) we observe,
\begin{align}
\forall \delta\in\left(0,\min_i\epsilon_i^0\right],~\exists ~\! t^K\geq t_0, ~\text{s.t.}~V_i(t)<\delta,~\forall t\geq t^K,~\forall i\notag
\end{align}
which concludes the proof, because of (\ref{E:cond_asymp}).
}

\addtolength{\textheight}{-12cm}   
                                  
\bibliographystyle{IEEEtran}
\bibliography{references}

\end{document}